\documentclass[12pt]{amsart}

\usepackage{amsfonts, amsmath, amssymb, amsthm}
\usepackage{changebar}

\newtheorem{lemma}{Lemma}
\newtheorem{theorem}{Theorem}

\newtheorem{proposition}{Proposition}

\newcommand{\nth}[1]{$#1 {\rm - th }$}

\newcommand{\ord}{{\rm ord} \ }

\newcommand{\zprk}{\Z_p\hbox{-rk}}
\newcommand{\zrk}{\Z\hbox{-rk} }

\newcommand{\Z}{\mathbb{Z}}
\newcommand{\ZM}[1]{\Z /( #1 \cdot \Z)}
\newcommand{\ZMs}[1]{(\Z / #1 \cdot \Z)^*}

\newcommand{\rg}[1]{\mbox{\bf #1}}
\newcommand{\eu}[1]{\mathfrak{#1}}
\newcommand{\id}[1]{\mathcal{#1}}
\newcommand{\Gal}{\mbox{ Gal }}
\newcommand{\Gals}{\mbox{\small Gal }}
\newcommand{\Ker}{\mbox{ Ker }}

\newcommand{\prk}{p - {\rm rank } \ }

\newcommand{\rf}[1]{(\ref{#1})}

\newcommand{\Norm}{\mbox{\bf N}}

\newcommand{\F}{\mathbb{F}}

\newcommand{\K}{\mathbb{K}}

\newcommand{\KH}{\mathbb{H}}

\newcommand{\KL}{\mathbb{L}}

\newcommand{\Q}{\mathbb{Q}}

\newcommand{\N}{\mathbb{N}}

\def\ra{\rightarrow}

 \newcommand{\ran}{\rangle}
\newcommand{\lan}{\langle}

\begin{document}
{\obeylines \small
\vspace*{-1.0cm}
\hspace*{3.5cm}Like a bird on the wire,
\hspace*{3.5cm}Like a drunk in a midnight choir
\hspace*{3.5cm}I have tried in my way to be free\footnote{Leonard Cohen: {\em Bird on the wire}.}.
\vspace*{0.4cm}
\hspace*{5.0cm} {\it To Stephen Smale, at his $80$ - th birthday}
\vspace*{0.5cm}
\smallskip
}
\title[Washington Heuristics] {Turning Washington's heuristics in
  favor of Vandiver's conjecture} 
\author{Preda Mih\u{a}ilescu}
\address[P. Mih\u{a}ilescu]{Mathematisches Institut der Universit\"at
  G\"ottingen} \email[P. Mih\u{a}ilescu]{preda@uni-math.gwdg.de}

\date{Version 2.0 \today}
\vspace{1.0cm}
\begin{abstract}
  A famous conjecture bearing the name of Vandiver states that $h_p^+
  = 1$ in the $p$ - cyclotomic extension of $\Q$. Heuristics arguments
  of Washington, which have been briefly exposed in \cite{La}, p. 261 and
\cite{Wa}, p. 158 suggest that the Vandiver conjecture should be
  false, if certain conditions of statistical independence are
  fulfilled. In this note we assume that Greenberg's conjecture is
  true for the \nth{p} cyclotomic extensions and prove an elementary
  consequence of the assumption that Vandiver's conjecture fails for a
  certain value of $p$: the result indicates that there are deep
  correlations between this fact and the defect $\lambda^- i(p)$,
  where $i(p)$ is like usual the irregularity index of $p$, i.e. the
  number of Bernoulli numbers $B_{2k} \equiv 0 \bmod p, 1 < k <
  (p-1)/2$. As a consequence, if one combines the various assumptions 
in Washington's heuristics, these turn, on base of the present result, into
an argument in favor of the Vandiver's conjecture.
\end{abstract}
\maketitle

\section{Introduction}
Let $p$ be an odd prime and $\K = \Q[ \zeta ]$ be the \nth{p}
cyclotomic field and $G = \Gal(\K/\Q)$. If $X$ is a finite abelian
group, we denote by $X_p$ its $p$ - Sylow group; let $A =
\id{C}(\K)_p$, the $p$ - Sylow subgroup of the class group
$\id{C}(\K)$ and $h^+, h^-$ the sizes of $A^+$ respectively $A^-$. In
a letter to Kronecker from 1857, Kummer refers to $p \nmid h^+$ as a
\textit{noch zu beweisender Satz}, a theorem yet to prove (see also
\cite{Wa}, p. 158). The fact was stated later as a conjecture by
Vandiver.

In \cite{La}, p. 261 Washington gives an heuristic argument which
suggests that there might be an asymptotic amount of $O(\log \log(N))$
of primes $p \leq N$ for which $\lambda(A_{\infty}^-) = i(p) +1$,
where $i(p)$ is the irregularity index of $p$, i.e. the number of
Bernoulli numbers $B_{2k}, 1 < k < (p-1)/2$ that vanish modulo $p$. In
\cite{Wa}, p.158, Washington starts with a naive argument, on base of
which the cyclotomic unit $\eta_{2k} := e_{2k} (1-\zeta)^{\sigma-1}$
(see below for the definition of the idempotents $e_j \in \F_p[ G ]$)
may be a \nth{p} power with probability $1/p$: this yields a
probability of more than one half, for the failure of Vandiver's
conjecture, so the argument is obviously too crude. Washington
considers then that the probabilities that a Bernoulli number vanishes
modulo $p$ and the one that the corresponding cyclotomic unit
$\eta_{2k}$ is a $p$-power are independent: this heuristic leads to a
frequence of $O(\log \log(n))$ primes $p < n$ for which the conjecture
fails.  As a consequence, various specialists in the field expect that
the conjecture should not always hold.  Our result in this note, shows
that if Vandiver's conjecture fails, then one has the additional
condition $\lambda^- > i(p)$. If one considers this condition also as
statistically independent (!) from the two conditions in Washington's
heuristiscs, then the same argument suggest that there may be $O(1)$
primes $p < n$ for which Vandiver's conjecture fails, thus possibly
none. \textit{Therefore the elementary result in this note may be
  understood as one that turns Washington's heuristics into an
  argument in favor of Vandiver's conjecture}.

In this paper we prove an elementary fact, which implies that the
failure of Vandiver's conjecture has an impact on the value of
$\lambda$, and thus two events which were supposed to be uncorrelated
in the heuristic approach: namely $h_p^+ \neq 0$ and $\lambda^- >
i(p)$ are not independent. No direct consequence can be drawn as to
the truth of the Kummer - Vandiver conjecture; however we have an
explicite theorem which indicates an unknown dependence, and also a
method of investigation which may be extended for the purpose of
investigating more possible consequences of the assumption that the
Kummer-Vandiver conjecture is false.

The result of this paper is the following:
\begin{theorem}
\label{vafal}
Let $p$ be an odd prime with irregularity index $i(p) = 1$. If $h_p^+
> 1$, then $\lambda^- \geq 2$.
\end{theorem}
Since $\lambda^- > 1$ is an implication of $h_p^+ > 1$, the two events
cannot be considered as independent events, each one with probability
$1/p$. But this implication can also suggest that the probability that
$\eta_{2k}$ is a \nth{p} power has rather the probability $1/p^2$ than
$1/p$, since it implies the vanishing of a higher order Bernoulli
number. Either way, we consider that our elementary result should
suggest that it is worthwhile to consider that Vandiver's conjecture
might be true, and pursue the investigation for the reasons why this
may be the case. For this purpose, the central idea of our proof can
be extended, with additional detail, to the general case, and this
shall be done in a subsequent paper.

Note that we restrict our analysis, for simplicity, to the case of
irregularity index $1$. However, this is the critical case in
Washington's heuristics, and if the assumption of ``statistical
independence'' is close to reality\footnote{Washington mentions
  explicitly that this is the crucial and critical in the various
  heuristics of this kind.}, then the probability of failure of the
conjecture for higher values of $i(p)$ can only be smaller, so the
argument stays valid. 

\section{Proof of the Theorem}
We let $\K = \Q[\zeta]$ be the \nth{p} cyclotomic extension and $\K_n
= \K[ \zeta^{1/p^n} ], n \geq 1$, the \nth{p^n} extension.  The galois
groups are
\begin{eqnarray*}
  G = \Gal(\K/\Q) & = & \{ \sigma_a \ : \ a = 1, 2, \ldots p-1, \ \zeta \mapsto \zeta^a \} \cong \ZMs{p}, \\
  G_n = \Gal(\K_n/\Q) & = & G \times \lan \tau \ran, \quad \tau(\zeta_{p^n}) = \zeta_{p^n}^{1+p},  
\end{eqnarray*}
so $\tau$ generates $\Gal(\K_n/\K)$, in particular.  If $g \in \F_p$
is a generator of $\ZMs{p}$, then $\sigma = \sigma_g$ generates $G$
multiplicatively. We write $\jmath \in G$ for complex
multiplication. For $\sigma \in G$ and $R \in \{ \F_p, \Z_p, \ZM{p^m}
\}$ we let $\varpi(\sigma) \in R$ be the value of the Theichm\"uller
character on $\sigma$; for $R = \F_p$ we may also write $\hat{\sigma}$
for this values. The orthogonal idempotents $e_k \in R[ G ]$ are
\[e_k = \frac{1}{p-1} \sum_{a=1}^{p-1} \varpi^k(\sigma_a) \cdot
\sigma_a^{-1}. \] If $X$ is a finite abelian $p$ - group on which $G$
acts, then $e_k(\Z_p)$ acts via its approximants to the \nth{p^m}
order; we shall not introduce additional notations for these
approximants. A fortiori, complex conjugation acts on $X$ splitting it
in the canonical plus and minus parts: $X = X^+ \oplus X^-$, with $X^+
= X^{1+\jmath}, X^- = X^{1-\jmath}$. The units of $\K$ and $\K_n$ are
denoted by $E, E_n$ and the cyclotomic units by $C, C_n$. The Iwasawa
invariants $\lambda,\lambda^-$ are related to the cyclotomic
$\Z_p$-extension $\K_{\infty} = \cup_n \K_n$ and $A_n =
(\id{C}(\K_n))_p$ are the $p$-parts of the ideal class groups of
$\K_n$. They form a projective sequence with respect to the relative
norms $N_{m,n} = \Norm_{\K_m/\K_n}, m > n \geq 1$ and $\rg{A} =
\varprojlim_n A_n$. We also write $A$ for $A_1$. We shall write for
simplicity $A(\KL) = (\id{C}(\KL))_p$ for the $p$-part of the class
group of an arbitrary number field $\KL$, so $A = A(\K)$, etc.

We fix now an odd prime $p$ such that
\begin{itemize}
\item[ 1. ] Greenberg's conjecture holds for $p$, so $A^+$ is finite
  and $\lambda^+ = 0$.
\item[ 2. ] Vandiver's conjecture fails for $p$.
\item[ 3. ] There is a unique irregular index $2k$ such that $A_{p-2k} = \varepsilon_{p-2k} A \neq \{ 1 \}$. Additionally $A_{2k} \neq \{ 1 \}$, as a consequence of 2.
\end{itemize}

Under these premises, we show that $\zprk( \varepsilon_{p-2k} \rg{A} >
1$, which is the statement of the theorem. We prove the statement by
contraposition, so we assume that $\zprk(\varepsilon_{p-2k} \rg{A}) =
1$. Since there is a unique irregular index, the minimal polynomial of
$\rg{A}$ is linear. Let $\KH_n/\K_n$ be the maximal $p$-abelian
unramified extensions. They split in plus and minus parts according to
$A_n = A_n^+ \oplus A_n^-$ and our assumption implies that
$\KH_n^+/\K_n$ are cyclic extensions of degree
\[ d_n := [ \KH_n^+ : \K_n ] = | A_n^+ |. \] We may also consider
$\KH_n^+$ as the compositum of $\K_n$ with the full $p$-part of the
Hilbert class field of $\K_n^+\subset \K_n$, the maximal reals
subextension of $\K_n$: thus $\KH_n^+$ is a canonical subfield,
corresponding by the Artin map to $A_n^+$. It follows that
$\KH_n^+/\K_n^+$ is an abelian extension, and thus $\KH_n^+$ is a CM
field (see also \cite{Wa}, Lemma 9.2 for a detailed proof).  There is
a canonic construction of radicals from $A_n^-$, such that $\KH_{n}
\cdot \K_m \subset \K_m[ (A_m^-)^{1/p^m} ]$ for sufficiently large
$m$. As a consequence of Greenberg's conjecture holding for $\K$,
there is an $n_0 \geq 1$ such that $| A_n^+ | = | A_{n_0}^+ |$ for all
$n \geq n_0$ and for such $n$, let $a_n \in A_n^-$ generated this
cyclic group. Let $\eu{Q} \in a_n$ and $\alpha_0 \in \K_n^{\times}$
with $(\alpha_0) = \eu{Q}^{\ord(a_n)}$; there is an $\alpha = \eta
\cdot \alpha_0^{1-\jmath}, \eta \in \mu_{p^n}$ which is well defined
up to roots of unity, such that $\KH_n^+ \subset \K_n[ \alpha^{1/p^n}
]$. The radical $B_n$ of $\KH_n^+$ is then the multiplicative group
generated by $\alpha$ and $(\K_n^{\times})^{d_n}$.

Since $\KH_n^+/K_n$ is cyclic, a folklore result, which we prove for
completeness in Lemma \ref{nohil} of the Appendix below, implies that
\begin{eqnarray}
\label{hiltriv}
A(\KH_n^+) = (\id{C}(\KH_n^+))_p = \{ 1 \}. 
\end{eqnarray}

A classical result, proved by Iwasawa \cite{Iw} in a general
cohomological language, states that for an arbitrary galois extension
$\KL/\F$ of finite number fields, there is a canonical isomorphism
\begin{eqnarray}
\label{iwa}
 H^{1}(\Gal(\KL/\F), E(\KL)) \cong \id{A}(\KL) , 
\end{eqnarray}
where $\id{A}(\F)$ are Hilbert's \textit{ambig} ideals, i.e. the
ideals of $\KL$ which are invariant under $\Gal(\KL/\F)$, factored by
the principal ideals of $\F$. These can be either totally ramified
ideals or ideals from $\F$ that capitulate completely (become
principal) in $\KL$. We shall in the sequel often consider the
homology groups $H^0, H^1$ for the unit groups. We can then write, for
simplicity
\[ H^i(\KL/\F) := H^{i}(\Gal(\KL/\F), E(\KL)) \quad \hbox{for} \quad i
= 0, 1. \]

The isomorphism above restricts also to one of $p$-parts of the
respective groups; furthermore, complex conjugation also induces
canonical isomorphisms of the plus and minus parts of $H^i$. The
extensions $\KH_n^+/\K_n$ being cyclic of degree $d_n$, the Herbrand
quotient is $d_n$ and thus
\[ H^1 (\KH_n^+/\K_n) = d_n \cdot H^0(\KH_n^+/\K_n) . \] We claim that
$\left(H^0(\KH_n^+/\K_n)\right)^+ = \{ 1 \}$. Indeed, the ambig ideals
in an unramified extension are capitulated ideals. In our case, since
$d_n = | A_n^+ |$ by definition, we have exactly $| \id{A}(\KH_n^+) |
= d_n^2$. This follows from the fact that the plus part capitulates
completely, while the minus part is cyclic too and generates the
radical of the extension. Consequently $\eu{Q}^{\ord(a_n)/d_n}$
becomes principal in $\KH_n^+$, which confirms that
\[ | \id{A}(\KH_n^+) | = d_n^2 \quad \hbox{ and } | \id{A}(\KH_n^+)
|^- = | \id{A}(\KH_n^+) |^+ = d_n. \] Therefore, $| H^0(\KH_n^+/\K_n)
| = d_n$. The roots of unity $\zeta_{p^n} \not \in
\Norm_{\KH_n^+/\K_n}(E(\KH_n^+))$: indeed, if $\zeta_{p^m} =
N(\delta)$ for $\delta \in E(\KH_n^+)$ and $m \leq n$, then
$\varepsilon = \delta/\overline{\delta}$ is well defined in the CM
field $\KH_n^+$ and it is a root of unity, by Dedekind's unit Theorem
-- so $\varepsilon \in \K_n$. Moreover, we have
$\Norm_{\KH_n^+/\K_n}(\varepsilon) = \varepsilon^{d_n} =
\zeta_{p^m}^2$. Since $p$ is odd, it follows that
$\mu_{p^n}/\mu_{p^n/d_n} \subset H^0(\KH_n^+/\K_n)$; by comparing
orders of the groups, we conclude that
\[ H^0(\KH_n^+/\K_n) = \mu_{p^n}/\mu_{(p^n/d_n)} =
\left(H^0(\KH_n^+/\K_n)\right)^-. \] We have proved:
\begin{lemma}
\label{h0v}
Notations being like above,
\[ \left(H^0(\KH_n^+/\K_n)\right)^+ = \{ 1 \} .\]
In particular
\begin{eqnarray}
\label{nsur}
\Norm_{\KH_n^+/\K_n}(E^+(\KH_n^+)) = E^+(\K_n), 
\end{eqnarray}
where for a CM field $\F$ we write $E^+(\F) = \{ e \cdot \overline{e}
: e \in E(\F) \}$.
\end{lemma}
In our case, $E^+$ are the real units and the units of $\K_n^+$,
resp. $\KH(\K_n^+) \subset \KH_n^+$; the prime $p$ is odd and we are
interested in $p$-parts, so the implicit exponent $2$ in the above
definition has no further consequences: the norm is surjective on the
real units in our class field.  Then $\KH_{n+1}^+ = \K_{n+1} \cdot
\KH_n^+$ and we have a commutative diagram of fields. By computing
$H^0(\KH_{n+1}^+/\K_n)$ in two ways, over the intermediate field
$\K_{n+1}$ and respectively over $\KH_n^+$, we obtain from \rf{nsur}
that
\begin{eqnarray}
\label{heq}
\left(H^0(\K_{n+1}/\K_n)\right)^+ = \left(H^0(\KH^+_{n+1}/\KH^+_n)\right)^+ .
\end{eqnarray}
The core observation of this proof is
\begin{proposition}
\label{metacyc}
Let $\lambda_n = 1-\zeta_{p^n}$. Then the ramified prime $\wp_n =
(\lambda_n) \subset \K_n$ above $p$ splits totally in $\KH^+_n$ in
$p$-principal ideals. Moreover, if $\nu \in \Gal(\KH_n^+/\K_n)$ is a
generator of this cyclic group, then there is a prime $\pi_n \in
\KH_n^+$ with $\Norm_{\KH_n^+/\K_n}(\pi_n) = \lambda_n$.
\end{proposition}
\begin{proof}
  Since $\wp_n = (\lambda_n)$ is principal, the Principal Ideal
  Theorem implies that it splits completely in the unramified
  extension $\KH_n^+/\K_n$ and since $A(\KH_n^+) = \{ 1 \}$ by
  \rf{nohil}, the primes above $\wp_n$ are $p$-principal. Let $E'(\F)$
  denote the $p$-units of the number field $\F$, i.e. the units of the
  smallest ring containing $E(\K)$ and in which all the primes above
  $p$ are invertible. In particular $E_n' = E_n[1/\lambda_n]$; it is
  customary to denote by $A'_n$ the $p$-part of the ideal class group
  of the $p$-integers. Since $\KH_n^+/\K_n$ splits the prime above $p$
  and $A_n^+ = (A'_n)^+$ and $H^0(\Gal(\KH_n^+/\K_n), E'(\KH_n^+)) =
  H^0(\KH_n^+/\K_n)$. In particular, the norm $\Norm_{\KH_n^+/\K_n} :
  E'(\KH_n^+) \ra E'(\K_n)$ is surjective, so there is a prime $\pi_n
  \in \KH_n^+$ mapping on $\lambda_n$.

  The proof of this proposition is made particularly simple by the use
  of \rf{nohil}. However, a more involved proof shows that the facts
  hold in more generality and the primes above $\lambda$ are principal
  in any subfield of the Hilbert class field $\KH_n$.
\end{proof}

As a consequence of the proposition, we see that $\id{A}(\KH_n^+) = \{
1 \}$. Indeed, by Lemma \ref{nohil}, the class group $A(\KH_n^+) = \{
1 \}$, and the only primes that ramify in $\KH_{n+1}^+/\KH_n^+$ lay
above $p$, so they are principal by Lemma \ref{metacyc}. There are
consequently no real ambig ideals in $\KH_n^+$. On the other hand, for
$n$ such that $| A^+_n | = | A_{n+1}^+ |$, the capitulation kernel
$P_n := \Ker(\iota_{n,n+1} : A_n^+ \ra A_{n+1}^+)$ is an $\F_p$-space
of dimension $d = \prk(\rg{A}^+) = \prk(A_n) = 1$.  We obtain a
contradiction with \rf{heq}, which shows that if $i(p) = 1$, then
either $\zprk(A^-) > 1$ or $h_p^+ = 1$. This completes the proof of
the Theorem.

\section{Three detailed variants of the proof}
Let $m$ be the first index for which $| A_m^(\K) | = | A_{m+1}^+(\K) | = | A^+(\K) |$.
Consider now $\KH_n^+ = \KH_m \cdot \K_n^+$ for $n \geq m$. Since $[\KH_n^+ : \K_n^+ ]
= q = | A(\K)^+ |$ for all $n > m$, we see that $\Ker( A^+(\KH_m) \ra A^+(\KH_{\infty}) = 1$.

Iwasawa's Theorem 12 in \cite{Iw1} implies that $H^1(\Gal(\KH_n/\KH_m), E'(\KH_n)) = 1$, and the fact that 
the primes above $p$ are principal leads by a computation of Herbrand quotients to 
$E(\KH_m) = N_{n,m}(E(\KH_n)) $ for all
$n > m$. We have 
\begin{eqnarray*}
 N_{\KH_n, \K^+_m}(E(\KH_n)) & = & N_{\KH_m/\K^+_m}(E(\KH_m)) = E(\K_m^+) \\
& = & N_{n,m}(N_{\KH_n/\K_n^+}(E(\KH_n)) = N_{n,m}(E(\K_n^+)). 
\end{eqnarray*}
It follows thus that $N_{n,m}(E(\K_n^+)) = E(\K_m^+)$ for all $n > m$. However, 
for sufficiently large $n$ we have $N_{n,m}(E(\K_n^+)) = C(\K_m^+)$ and this would require that 
$E(\K_m^+) = C(\K_m^+)$, which is the contradiction above.

We give below\footnote{I thank Hiroki Takahashi for some critical questions and an interesting dialogue,
which suggested that the details provided in this section may be useful for readers interested to understand
better the specific conditions of the \nth{p} cyclotomic extensions, compared for instance with some
real quadratic extensions in which Greenberg's conjecture is known to hold, yet $A^+ = p$ and $A^-$ is cyclic.}
 some additional detail about the units in $\KH_n$ which lead to two variants of this
proof and reveal interesting details.
\subsection{Metacyclotomic units}
Let $E(\KH_n)$ be the units of $\KH_n^+$ and $E(\K_n),
C(\K_n)$ the units, resp. cyclotomic units of $\K_n$. We shall
construct some norm coherent sequences of units in $E(\KH_n)$ and
relate them to the cyclotomic. For this, fix some large integer
$N$. We assume for simplicity that the primes above $p$ in $\KH_n$ are
principal and derive some specific units on base of this
assumption. Since we are only interested in $p$-parts, raising to a
power coprime to $p$ can be neglected.

We shall denote the norms $\Norm_{\KH_n/\KH_m} = \Norm_{\K_n/\K_m}$ by
$N_{n,m}$ and let $\id{N} = \Norm_{\KH_n/\K_n}$ for all $n$
sufficiently large. We choose $\pi_{2N} \in \KH_{2N}$ a prime with
$\id{N}(\pi_{2N}) = 1-\zeta_{p^{2N}}$ and let $\pi_n =
N_{2N,n}(\pi_{2N})$ for all $n < 2N$. The inertia group $I(\pi_{2N})
\cong \Gal(\K_{2N}/\Q)$.  We let $\sigma', \tau' \in I(\pi_{2N})$ be
lifts of $\sigma, \tau \in \Gal(\K_{2N}/\Q)$, the generators of the
subgroups with $p-1$ and $p^{2N-1}$ elements, respectively. We
identify $T = \tau'-1$ a lift of $T$ to $\Gal(\KH_{2N}/\KH_m)$, where
$m$ is an integer such that $| A^+(\K_{m}) | = | A^+(\K_{m+1}) | = q$,
say.

We fix an other large integer $M >> N$ and define the lifted idempotents 
\[ \varepsilon'_{j} = \frac{1}{p-1} \sum_{\varsigma \in < \sigma' > }
\omega^j(\varsigma) \varsigma^{-1} \in \Z[ \Gal(\KH_n/\Q) ], \] where
$\omega : < \sigma' > \ra \Z$ approximates the Teichm\"uller character
to the power $p^M$. We also let 
\[ \tilde{\varepsilon}_{j} = \frac{1}{p-1} \sum_{\varsigma \in < \sigma' > }
\varpi^j(\varsigma) \varsigma^{-1} \in \Z_p[ \Gal(\KH_n/\Q) ], \]
with $\varpi$ the $p$-adic Teichm\"uller character.

Then we define $E^{(2j)}(\K_n^+) = \varepsilon_{2j} E(\K_n^+)$, so 
$\tilde{\varepsilon}_2j(E(\K_n^+)/E^{p^M}(\K_n^+)) = E^{(2j)}(\K_n^+)/(E^{(2j)}(\K_n^+))^{p^M}$. 
Moreover, $m$ verifies for all $n > m$  
\begin{eqnarray}
\label{stabe}
E(\K_n) = C(\K_n) \cdot E(\K_m), \quad \forall n > m.
\end{eqnarray}
Indeed, this condition is equivalent to $[ E(\K_n) : C(\K_n) ] = |
A_n^+ | = [ E(\K_m) : C(\K_m) ] = | A_m^+ |$ for all $n > m$, a
condition which is implied by our assumption (e.g. Fukuda). It follows
also that $[ E(\K_l) : C(\K_l) ] = | A^+_l | = p | A_{l-1}^+ | $ for
$1 < l \leq m$, and if $\varepsilon_{l} \in e_{2k} E(\K_l)$ generates
this component up to \nth{p^M} powers, then
\[ E(\K_n) = \lan \varepsilon_m, \eta_n \ran_{\Z[ T ]} . \]

We show 
\begin{lemma}
\label{metac}
Notations being like above, there is a module $\overline{C}(\KH_n) \subset E(\KH_n)$, of finite index and norm coherent for all $n \leq 2 N$. Moreover, if $m$ is the smallest index such that
$| A_m^+ | = |A^+_{m+1} |$, then 
\begin{eqnarray}
\label{m}
\label{stab2}
E(\KH_n)  =  E(\KH_{m}) \cdot \overline{C}(\KH_n) \quad \forall n \geq m.
\end{eqnarray}
\end{lemma}
\begin{proof}
Let $\eta_n = (\pi_n \cdot \overline{\pi_n})^{\sigma'-1}$ with $\pi_n$ defined above, let
$\delta_n = (\pi_n \cdot \overline{\pi_n})^T$ and $C(\KH_n) =
\eta_n^{\Z[ \Gals(\KH_n/\Q) ]}$ be the $\Z$-module generated by
$\eta_n$, for all\footnote{We define the \textit{metacyclotomic} units
  only for $n < N$, in order to make sure that they are norms from a
  large extension above $\KH_n$. By avoiding the limit process $\cap_N
  N_{N,n}(E'(\KH_n))$ we obtain in addition explicit uniformizors
  $\pi_n$, which are norm coherent.}  $n < N$.  By definition,
$\id{N}(C(\KH_n)) = C(\K_n)$ and the modules $C(\KH_n)$ are norm
coherent. In particular, if $e \in C(\KH_l)$ then there is a unit $e'
\in C(\KH_{2N})$ with $e = N_{2N,l}(e')$.

Let further $(a_n)_{n \in \N} \in A^+_n \setminus (A_n^+)^p$ be a norm
coherent sequence of classes; thus $a_n^q = 1$ for all $n \geq m$ and
$q = \exp(A^+)$. Let $\eu{Q} \in a_{2N}$ be a totally split prime and
$\gamma \in \KH_{2N}$ be such that $\eu{Q} \id{O}(\KH_{2N}) =
(\gamma)$; if $\eu{Q}_n = N_{2N,n}(\eu{Q})$ and $\gamma_n =
N_{2N,n}(\gamma)$, then $\eu{Q}_n \id{O}(\KH_n) = (\gamma_n)$ for all
$n < 2N$. We let $c_n = \gamma_n^s \in E(\KH_n)$, a norm coherent
sequence of units for $n \leq 2N$. Let $\id{C}(\KH_n) = c_n^{\Z[ s ]}
\cdot \delta_n^{\Z[ T, s ]}$. Then
$\id{C}(\KH_n)$ also form a norm coherent sequence and we define
$\overline{C}(\KH_n) = C(\KH_n) \cdot \id{C}(\KH_n)$, which is a
further norm coherent sequence of units in $E(\KH_n)$. 

We note that $\Ker(\id{N} : E(\KH_n) \ra E(\K_n^+)) = E^s(\KH_n) \cdot \id{C}(\KH_n)$. Letting 
$\id{E}(\KH_n) = E(\KH_n)/( E^s(\KH_n) \cdot \id{C}(\KH_n))$, it follows from
\rf{nsur} that
\[ \id{E}(\KH_n) \cong E(\K_n), \] 
and the norm $\id{N}$ is an isomorphism between these two modules. As a consequence, 
if $c_0 \in \id{C}(\KH_n)$ generates this module and $e_{2j} \in E(\KH_n), j = 0, 1, \ldots, 
\frac{p-3}{2} $ have non trivial image in $\id{E}$ and map to a set of units 
$d_j ; j = 0, 1, \ldots, (p-3)/2$ which generate $E(K_n^+)$ as a $\Z[ T ]$-module and such 
that $d_j$ generates $\varepsilon_{2j} E(\K_n^+)/E(\K_n^+)^{p^M}$ for fixed, arbitrarily 
large $M > 0$, then 
\[  \lan c_0; e_0, \ldots, e_{(p-3)/2} \ran_{\Z[ T , s ]} = E(\KH_n). \]

We let $e_n \in E(\KH_n)$ be such that $\id{N}(e_n)$
generates $\varepsilon_{2j}E(\K_n)/E^{p^M}(\K_n)$ for some fixed, large $M$. 
We then consider the systems
\begin{eqnarray*}
  \id{H}_n & = & \left\{ c_n^{s^j} \ : \ j = 0, 1, \ldots, p-2 \right\} \bigcup \left\{ \delta_n^{s^j T^l} \ : \
    j = 0, 1, \ldots, p-1; l = 0, \ldots, p^{n-1}-1 \right\}
  \\ & & \bigcup \left\{ \varepsilon'_{2j} \eta_n^{s^i \cdot T^l} \ : \
    i = 0, 1, 2, \ldots, p-1; j = 1, 2, \ldots, \frac{p-3}{2}; l = 0, \ldots, p^{n-1}-1 \right\}, \\
  \overline{\id{H}}_n & = & \left\{ c_n^{s^j} \ : \ j = 0, 1, \ldots, p-2 \right\} \bigcup \left\{ \delta_n^{s^j T^l} \ : \
    j = 0, 1, \ldots, p-1; l = 0, \ldots, p^{n-1}-1 \right\}
  \\ & & \bigcup \left\{ (e_n^{(2j)})^{s^i \cdot T^l} \ : \
    i = 0, 1, 2, \ldots, p-1; j = 1, 2, \ldots, \frac{p-3}{2}; l = 0, \ldots, p^{n-1}-1 \right\}.
\end{eqnarray*}
One verifies that $| \id{H}_n | = \zrk(E(\KH_n))$; an application of
the Nakayama lemma to $\overline{C}(\KH_n)/\overline{C}(\KH_n)^{p^M}$
implies that the system is also a generating system for
$\overline{C}(\KH_n)$. Moreover $\overline{\id{H}}_n$ generates
$E(\KH_n)/E^{p^M}(\KH_n)$ for $n \leq m$.  These systems are
reminiscent of Hilbert's relative units in his Theorem 91.  It follows
that $\overline{C}(\KH_n)$ has finite index in $E(\KH_n)$ for all $n <
2N$.

Assume now that \rf{m} is false, and there is some $m' > m$
and $e_{m'} \in E(\KH_{m'}) \setminus E(\KH_{m}) \cdot
\overline{C}(\KH_{m'})$; we can assume without restriction of
generality that $e_{m'} \in E^{(2j)} (\KH_{m'})$ (see below for the definition
of the \textit{components} $E^{(2j)} (\KH_{m'})$), since at least one
component verifies the condition. Then, by \rf{stabe}
\[ \id{N}(e_{m'}) \in E(\K_{m}) \cdot \overline{C}(\K_{m'}) =
\id{N}(E(\KH_{m}) \cdot \overline{C}(\KH_{m'})). \] There is thus a
unit $e \in E(\KH_{m}) \cdot \overline{C}(\KH_{m'})$ such that $e_{m'}
= e \cdot w, w \in \Ker(\id{N} : E(\KH_{m'}) \ra E(\K_{m'}))$.  Since
$\id{C}_n$ are norm coherent and $(1-\varepsilon'_0) \Ker(\id{N} :
E(\KH_{m'}) \ra E(\K_{m'})) = (1-\varepsilon'_0) E(\KH_{m'})^s$, we
can assume that $w \in (1-\varepsilon'_0) E(\KH_{m'})$, so $w =
(e_{m'}')^s$ with $e_{m'}' \in (1-\varepsilon'_0) E(\KH_{m'})$. By
repeating the same argument, we find that $e_{m'}' = e'' \cdot w'$,
where $w' \in \Ker(\id{N} : E(\KH_{m'}) \ra E(\K_{m'}))$. It follows
by induction that $e_{m'} \in E(\KH_{m}) \cdot \overline{C}(\KH_{m'})
\cdot E(\KH_{m'})^{s^u}$ for any $u > 0$.  Since $s^p = p v, v \in
(\Z_p[ s ])^{\times}$, it follows that $e_{m'} \in E(\KH_{m}) \cdot
\overline{C}(\KH_{m'}) \cdot E(\KH_{m'})^{p^{M'}}$ for arbitrary $M' >
0$. But the index $[ E(\KH_{m'}) : \overline{C}(\KH_{m'}) ]$ is
finite, so the claim must be true.  Therefore $m$ is defined as the
smallest index after which $| A_n^+ |$ stabilizes 
\end{proof}

We obtain the following explicite variant of the first proof given above.
Since $E(\KH_m) \neq \overline{C}(\KH_m)$ while the quotient of these modules is finite, 
there is a unit $e \in E(\KH_m) \setminus \overline{C}(\KH_m)$, such that $e^p \in \overline{C}(\KH_m)$.
Consequently, there is a metacyclotomic unit $c_{m+1} \in \overline{C}(\KH_{m+1})$ such that
$N_{m+1,m}(c_{m+1}) = e^p$. Let $d = c_{m+1}/e \not \in \overline{C}(\KH_{m+1})$, by choice
of $e$. By definition, $N_{m+1,m}(d) = 1$ and thus Hilbert 90 implies that
$d = \gamma^{\omega_m}$ for some $\gamma \in \KH_{m+1}$. The ideal $(\gamma)$ is ambig, and since
$A^+(\KH_{m+1}) = 1$, it follows that $\gamma$ must be a product of principal ramified primes.
Thus $\gamma = \pi_{m+1}^x \cdot \delta$, for some $x \in \Z[ \Gal(\KH_{m+1}/\Q) ]$ and $\delta \in E(\KH_{m+1}$.
But then $\pi_{m+1}^{x \omega_{m+1}} \in \overline{C}(\KH_{m+1})$ by definition of $\overline{C}$.
Moreover, the structure identity \rf{m} implies that $\delta^{\omega_m} \in \overline{C}(\KH_{m+1})^{\omega_m} \subset
\overline{C}(\KH_{m+1})$. Altogether, we find that $d \in \overline{C}(\KH_{m+1})$ and thus
$e \in \overline{C}(\KH_{m+1})$ too. This contradicts the choice $e \not \in \overline{C}(\KH_m)$, given 
the \rf{m}: units which are metacyclotomic in $\KH_{m+1}$ must be so already in $\KH_m$. Note that this
contradiction makes explicite the argument of Iwasawa used in the first proof.

\subsection{Local units and components}
We consider the structure of the local units $U(\K_n), U(\KH_n)$. For
the first, it is known that $U(\K_n^+) = \bigoplus_k \varepsilon_{2k}
x_{n}^{\Lambda}$ are $\Lambda$-cyclic modules, generated by some $x_n
\in U(\K^+_n)$ (e.g. \cite{Wa}, \S 8).  Since $p$ is totally split, we
also have $U(\KH_n) \cong (U(\K_n^+))^p$, as cartesian product.  Let
$\xi_n = (x_n, 1, 1, \ldots, 1)$ under the Chinese Remainder Theorem,
with the first component being the projection to $\KH_{n,\wp_n}$ and
$\wp_n$ a ramified prime above the initially chosen $\wp =
(\pi)$. Then $\tilde{\varepsilon}_{2j} \xi_n$ defines the $2j$-component of
$U(\KH_n)$ as a $\Lambda[ s ]$-module. Since $\nu$ acts transitively on 
the projections in $U(\KH_n)$, it follows that these modules are also \textit{canonically}
given by 
\begin{eqnarray}
\label{locun}
U_{2j} (\KH_n) = (U(\K_n^+))^p
\end{eqnarray}
where the right hand side is a cartesian of copies corresponding to the completions at the
various primes above $p$. The action of $\Gal(\K_n^+/\Q)$ by conjugation on $\nu$ implies
\begin{eqnarray*}
\nu^{\sigma} = \nu^{\varpi(\sigma)^{2k}}, \quad \nu^{\tau} = \nu^{p+1},
\end{eqnarray*}
with $\varpi$ being the Teichm\"uller character.
From this, one verifies without using the construction above, that $(U(\K_n^+))^p$ is a canonic $\Z_p[ \Gal(\KH/\Q) ]$-module which well defined by \rf{locun}. We say that $U_{2j}(\KH_n)$ is the \nth{2j} component of $U(\KH_n)$ and have
\[  U(\KH_n) = \bigoplus_{j=0}^{(p-3)/2} U_{2j}(\KH_n). \]

Let $\rg{N}_l = \cap_n N_{n,l}(U(\KH_n))$ for $l \geq 1$. Then class
field theory implies that $U(\KH_l)/\rg{N}_l \cong \prod
U^{(1)}(\Z_p)$ is the product of the $p$ copies of $U^{(1)}(\Z_p)$ in
$U(\KH_l)$.  Since $\Z_p \subset U_0(\K_n^+)$ it follows from the definition \rf{locun} that
$ U(\KH_l)/\rg{N}_l  \hookrightarrow U_0(\KH_n)$, and therefore
\begin{eqnarray}
\label{nl}
U_{2j}(\KH_m) \subset \rg{N}_m \quad \hbox{for all} \quad j \neq 0. 
\end{eqnarray}

We identify the units $E(\KH_n) $ with their \textit{diagonal embedding} $E(\KH_n)  \hookrightarrow U(\KH_n)$. Since $E(\KH_n)/E(\KH_n)^{p^M}$ is a $\Z_p$-module, 
we can use the components of the local units for defining the 
\textit{components} of global units by
\begin{eqnarray*}
E^{(2j)} (\KH_n) & = & E(\KH_n) \cap \left(U_{2j}(\KH_n) \cdot E(\KH_n)^{p^M} \right), \\
C^{(2j)}(\KH_n) & = & E^{(2j)} (\KH_n) \cap \overline{C}(\KH_n)
\end{eqnarray*}
By definition of $U_{2j}$, we have $E^{(2l)}(\KH_n) \cap E^{(2j)}(\KH_n) = E(\KH_n)^{p^M}$ 
for $l \neq j$.

Let now $e \in E^{(2k)}(\KH_m) \setminus
N_{m+1,m}(E^{(2k)}(\KH_{m+1}))$. Note that by \rf{stab2}, the
existence of such a unit is granted. But then, by applying $\id{N}$ we
deduce that $\varepsilon_{2k} E(\K_m) =
\varepsilon_{2k}C(\K_m)$. However, our initial assumption implies that
$\varepsilon_{2k} A^+_m \neq \{ 1 \}$ and the main conjecture thus
leads to $[ \varepsilon_{2k} E(\K_m) : \varepsilon_{2k}C(\K_m) ] \neq
1$. Therefore, we conclude that there is unit $e$ with the claimed
properties. The Hasse norm principle implies by the above that $e \in
N_{m+1,m}(\KH_{m+1}^{\times} )$. Let thus $x \in \KH_{m+1}^{\times}$
with $N_{m+1,m}(x) = e$.  Then $(x) = \eu{X}^{\omega_m}$, since $H^1(
\lan \nu \ran, I(\KH_n))$ vanishes for the ideals in the cyclic
extension $\KH_n/\K_n$.  Moreover, $\eu{X}$ is not $p$-principal:
otherwise, if $(t,p) = 1$ and $\eu{X}^t = (\xi)$ is principal, then
$x^t = \xi^{\omega_m} \cdot e_1$ for some unit $e_1 \in E(\KH_{m+1})$. If the
$u t ´+ vp = 1$ we have
\[ e = N_{m+1,m}(e^v \cdot \xi^{\omega_m} e_1) = N_{m+1,m}(e^v \cdot e_1) \in
N_{m+1,m}E(\KH_{m+1}), \] and we had assumed $e \not \in
N_{m+1,m}E(\KH_{m+1})$. Consequently $A(\KH_{m+1}) \neq \{ 1 \}$,
which is a contradiction to the Lemma \ref{nohil} that completes the
proof.

Note that universal norms behave differently in the positive components and in the zero component, the latter
being by \rf{nl} the one which localizes the norm residue defects along the cyclotomic $\Z_p$-extension of 
any base field.
This fact may indicate the major distinction between zero components and positive indexed components.

\section{Appendix}
For the sake of completeness, we give a proof of the following 
\begin{lemma}
\label{nohil}
Let $\K$ be a number field and $A$ be the $p$ - part of its class
group, while $\KH$ is the $p$ - part of its Hilbert class group. If
$A$ is cyclic, then $A(\KH) := \id{C}(\KH)_p = \{ 1 \}$.
\end{lemma}
\begin{proof}
  Since $A$ is cyclic, $\Gal(\KH/\K) \cong A$ is a cyclic group and
  the ideals of a generating class $a \in A$ are inert and become
  principal in $\KH$. Let $\sigma = \varphi(a) \in \Gal(\KH/\K)$ be a
  generator and $s = \sigma - 1$. Suppose that $b \in A(\KH) \setminus
  A(\KH)^{(s, p)}$ is a non trivial class and let $\eu{Q} \in b$ be an
  ideal above a rational prime $\eu{q} \subset \K$, which splits
  completely in $\KH/\K$: such a prime must exist, by Tchebotarew's
  Theorem. If $b' = [ \eu{q} ]$, then the order of $b'$ must be a
  power of $p$, since this holds for $\eu{Q}$; thus $b' \in A(\K) =
  \lan a \ran$. But we have seen that the ideals from $a$ capitulate
  in $\KH$, so $b = 1$, in contradiction with our choice. Therefore
  $b' = 1$ and thus $\Norm_{\KH/\K}(b) = 1$. Furtw\"{a}ngler's Hilbert
  90 Theorem for ideal class groups in cyclic extensions \cite{Fu} says
  that $\Ker(\Norm : A(\KH) \ra A(\K)) \subset A(\KH)^s$, and this
  implies that $b \in A(\KH)^s$, which contradicts the choice of $b$
  and completes the proof.
\end{proof}

\end{document}